\documentclass[aps,10pt,notitlepage]{revtex4-1}

\usepackage{amsthm}
\usepackage{amsfonts}
\usepackage{amssymb}
\usepackage{amsmath}
\usepackage{mathtools}
\usepackage{mathrsfs}

\newtheorem{lemma}{Lemma}
\newtheorem{proposition}{Proposition}

\allowdisplaybreaks
\begin{document}
\title{On an explicit representation of central $(2k+1)$-nomial coefficients}

\author{Michelle Rudolph-Lilith}
\email[Electronic address: ]{rudolph@unic.cnrs-gif.fr}
\author{Lyle E. Muller}
\affiliation{
Unit\'e de Neurosciences, Information et Complexit\'e (UNIC) \\
CNRS, 1 Ave de la Terrasse, 91198 Gif-sur-Yvette, France}

\date{\today}

\begin{abstract}
We propose an explicit representation of central $(2k+1)$-nomial coefficients in terms of finite sums over trigonometric constructs. The approach utilizes the diagonalization of circulant boolean matrices and is generalizable to all $(2k+1)$-nomial coefficients, thus yielding a new family of combinatorical identities.
\end{abstract}

\maketitle

% ============================================================================= %
% ============================================================================= %

\section{Introduction}
\label{S_Introduction}

In the first volume of \textit{Monthly} in 1894, De Volson Wood asked the question \emph{``An equal number of white and black balls of equal size are thrown into a rectangular box, what is the probability that there will be contiguous contact of white balls from one end of the box to the opposite end?''} \cite{Wood94}. Though this was one of the first and most clearly visualizable examples in a set of combinatorical problems which later would give rise to the field of percolation theory \cite{Grimmet89}, to this date it still eludes a definite solution. 

Mathematically, this problem is linked to counting walks in random graphs, and recently an approach was proposed which translates this combinatorically hard problem into taking powers of a specific type of circulant boolean matrices \cite{Rudolph-LilithMuller14}. Interestingly, here a specific type of sum over powers of fractions of trigonometric functions with fractional angle appears, specifically $\sum_{l=1}^{m-1} (\sin[kl \pi/m] / \sin[l \pi/m])^n$ for any given $k,m \in \mathbb{N}: m > 1, 1 \leq k \leq \lfloor m/2\rfloor$. The latter shows a striking similarity to the famous Kasteleyn product formula for the number of tilings of a $2n \times 2n$ square with $1 \times 2$ dominos \cite{Kasteleyn61, TemperleyFisher61}. 

Numerically one can conjecture that these constructs yield integer numbers which are multiples of central $(2k+1)$-nomial coefficients and, thus, provide an explicit representation of the integer sequences of multinomial coefficients in terms of finite sums over real-valued elementary functions. Although various recursive equations addressing these coefficients do exist, and the Almkvist-Zeilberger algorithm \cite{AlmkvistZeilberger90} allows for a systematic derivation of recursions for multinomial coefficients in the general case, no such explicit representation has yet been proposed.

Here, we provide a simple proof of the aforementioned conjecture and propose an explicit representation of central $(2k+1)$-nomial coefficients. To that end, let 
\begin{equation}
\label{Eq_Px}
P(x) = 1 + x + x^2 + \cdots + x^{2k}
\end{equation}
be a finite polynomial of even degree $2k, k \in \mathbb{N}, k \geq 1$ in $x \in \mathbb{Q}$. Using the multinomial theorem and collecting terms with the same power in $x$, the $n$th power of $P(x)$ is then given by
\begin{equation}
P(x)^n 
 = ( 1 + x + x^2 + \cdots + x^{2k} )^n
 = \sum\limits_{l=0}^{2kn} p_l^{(n)} x^l
\end{equation}
with 
\begin{equation}
\label{Eq_pl}
p_l^{(n)} = \sum\limits_{\mathclap{\substack{n_i \in [0,n] \forall i \in [0,2k] \\ n_0+n_1+\cdots+n_{2k}=n \\ n_1+2n_2+\cdots+2kn_{2k}=l}}} \hspace*{7mm} \binom{n}{n_0,n_1,\cdots,n_{2k}} .
\end{equation}
The central $(2k+1)$-nomial coefficients $M^{(2k,n)}$ are then given by $M^{(2k,n)} = p_{kn}^{(n)}$.

% ============================================================================= %
% ============================================================================= %

\section{A trace formula for central $(2k+1)$-nomial coefficients}

Consider the $(2kn+1) \times (2kn+1)$ circulant matrix
\begin{equation}
\mathbf{A} 
= \mathrm{circ} \big\{ ( \, \overbrace{1,\ldots,1}^{2k+1},0,\ldots,0 \, ) \big\}
= \mathrm{circ} \left\{ 
\Big( \sum\limits_{l=0}^{2k} \delta_{j,1+l \, \mathrm{mod} (2kn+1)} \Big)_j
\right\} .
\end{equation}
Multiplying $\mathbf{A}$ by a vector $\mathbf{x} = (1,x,x^2,\ldots,x^{2k}) \in \mathbb{Q}^{2k+1}$ will yield the original polynomial as the first element in the resulting vector $\mathbf{A} \mathbf{x}$. Similarly, taking the $n'$th ($n' \leq n$) power of $\mathbf{A}$ and multiplying the result with $\mathbf{x}$ will yield $P(x)^{n'}$ as first element, thus $\mathbf{A}^{n'}$ will contain the sequence of multinomial coefficients $p_l^{(n')}$ in its first row. Moreover, as the power of a circulant matrix is again circulant, this continuous sequence of non-zero entries in a give row will shift by one column to the right on each subsequent row, and wraps around once the row-dimension $(2kn+1)$ is reached. This behavior will not change even if one introduces a shift by $m$ columns of the sequence of 1 in $\mathbf{A}$, as this will correspond to simply multiplying the original polynomial by $x^m$. Such a shift, however, will allow, when correctly chosen, to bring the desired central multinomial coefficients on the diagonal of $\mathbf{A}^{n'}$.

We can formalize this approach in the following

% ----------------------------------------------------------------------------- %

\begin{lemma}
\label{L_CMCtr}
Let
\begin{equation}
A^{(m)} = \mathrm{circ} \left\{ 
\Big( \sum\limits_{l=0}^{2k} \delta_{j,1+(m+l) \mathrm{mod} (2kn+1)} \Big)_j
\right\}
\end{equation}
with $m \in \mathbb{N}_0$ be circulant boolean square matrices of dimension $2kn+1$ with $k,n \in \mathbb{N}$ and $k,n \geq 1$. The central $(2k+1)$-nomial coefficients are given by
\begin{equation}
\label{Eq_Lemma1}
M^{(2k,n)} = \frac{1}{2kn+1} \mathrm{Tr} \big[ A^{(2kn-k)} \big]^n .
\end{equation}
\end{lemma}

% ----------------------------------------------------------------------------- %

\begin{proof}
Let $B = \mathrm{circ} \big\{ (0,1,0,\ldots,0) \big\}$ be a $(2kn+1)\times(2kn+1)$ cyclic permutation matrix, such that
\begin{eqnarray}
\label{Eq_Brules}
B^0 & \equiv & I = B^{2kn+1} \nonumber \\
B^m & = & B^r \quad \text{ with } r = m \, \text{mod}(2kn+1) \nonumber \\
B^n B^m & = & B^{(nm) \, \text{mod}(2kn+1)} ,
\end{eqnarray}
where $I$ denotes the $(2kn+1)$-dimensional identity matrix. The set of powers of the cyclic permutation matrix, $\{ B^{m} \}, m \in [0,2kn+1]$, then acts as a basis for the circulant matrices $A^{(m)}$. 

Let us first consider the case $m=0$. It can easily be shown that 
$$
A^{(0)} = I + \sum\limits_{l=1}^{2k} B^{l} . 
$$
Applying the multinomial theorem and ordering with respect to powers of $B$, we have for the $n$th power of $A^{(0)}$
\begin{eqnarray}
\label{Eq_A0n}
\big( A^{(0)} \big)^n
& = & \sum\limits_{\mathclap{n_0+n_1+\cdots+n_{2k}=n}} \hspace*{8mm} \binom{n}{n_0,n_1,\cdots,n_{2k}} I^{n_0} B^{n_1+2n_2+\cdots+2kn_{2k}} \nonumber \\
& = & b_0^{(n)} I + \sum\limits_{l=1}^{2kn} b_l^{(n)} B^l \nonumber \\
& \equiv & \text{circ} 
\Big\{
\big( b_0^{(n)}, b_1^{(n)}, \ldots , b_{2kn}^{(n)} \big)
\Big\} ,
\end{eqnarray}
where $b_l^{(n)} = p_l^{(n)}$.   
 
For $m>0$, we have
$$
A^{(m)} 
 = B^m + \sum\limits_{l=1}^{2k} B^{m+l} \\
 \equiv B^m \Big( I + \sum\limits_{l=1}^{2k} B^{l} \Big), 
$$
and obtain, with (\ref{Eq_A0n}), for the $n$th power
\begin{equation}
\label{Eq_Amn1}
\big( A^{(m)} \big)^n = B^{mn} \Big( b_0^{(n)} I + \sum\limits_{l=1}^{2kn} b_l^{(n)} B^l \Big) .
\end{equation}
Using (\ref{Eq_Brules}), the factor $B^{mn}$ shifts and wraps all rows of the matrix to the right by $mn$ columns, so that
\begin{eqnarray*}
\label{Eq_Amn}
\big( A^{(m)} \big)^n 
& = & b_0^{(n,m)} I + \sum\limits_{l=1}^{2kn} b_l^{(n,m)} B^l \\
& \equiv & \text{circ} 
\Big\{
\big( b_0^{(n,m)}, b_1^{(n,m)}, \ldots , b_{2kn}^{(n,m)} \big)
\Big\} ,
\end{eqnarray*}
where $b_l^{(n,m)} = b_{(l-nm) \text{mod}(2kn)}^{(n)}$ with $b_l^{(n)} = p_l^{(n)}$ given by (\ref{Eq_pl}). 

For $m=0$, the desired central multinomial coefficients can be found in column $(kn)$, i.e. $M^{(2k,n)} = b_{kn}^{(n)}$. Observing that 
$$
(l-n(2kn-k))\text{mod}(2kn) \equiv (l+kn)\text{mod}(2kn) \, ,
$$
a shift by $m=2kn-k$ yields
\begin{equation}
M^{(2k,n)} = b_{(l+kn)\text{mod}(2kn)}^{(n)} = b_{l}^{(n,2kn-k)} .
\end{equation}
That is, setting $l=0$, the central $(2k+1)$-nomial coefficients $M^{(2k,n)}$ reside on the diagonal of $\big(A^{(2kn-k)}\big)^n$. Taking the trace of $\big(A^{(2kn-k)}\big)^n$ thus proves (\ref{Eq_Lemma1}).
\end{proof}

With Lemma~\ref{L_CMCtr}, the sequences of central $(2k+1)$-nomial coefficients $M^{(2k,n)}$ are given in terms of the trace of powers of the $(2kn+1) \times (2kn+1)$-dimensional circulant boolean matrix
\begin{eqnarray}
\label{Eq_ACMC}
A^{(2kn-k)}
& = & \mathrm{circ} 
\big\{
(1,\overbrace{1,\ldots,1}^{k},0,\ldots,0,\overbrace{1,\ldots,1}^{k})
\big\} \nonumber \\
& = & \mathrm{circ} \left\{ 
\Big( \sum\limits_{l=0}^{k} \delta_{j,1+l} + \sum\limits_{l=0}^{k-1} \delta_{j,1+2kn-l} \Big)_j
\right\} .
\end{eqnarray}
This translates the original problem not just into one of matrix algebra, but, in effect, significantly reduces its combinatorical complexity to finding powers of circulant matrices.

% ============================================================================= %
% ============================================================================= %

\section{A sum formula for central $(2k+1)$-nomial coefficients}

Not only can circulant matrices be represented in terms of a simple base decomposition using powers of cyclic permutation matrices (see above), but circulant matrices also allow for an explicit diagonalization \cite{Davis70}. The latter will be utilized to prove the main result of this contribution, namely 

% ----------------------------------------------------------------------------- %

\begin{proposition}
\label{P_Msin}
The sequence of central $(2k+1)$-nomial coefficients $M^{(2k,n)}$ with $k,n \in \mathbb{N}, k,n > 0$ is given by
\begin{equation}
\label{Eq_Msin}
M^{(2k,n)} = \frac{1}{2kn+1} 
\left\{
(2k+1)^n + \sum\limits_{l=1}^{2kn} \left( \frac{\sin\big[\frac{(2k+1)l}{2kn+1} \pi\big]}{\sin\big[\frac{l}{2kn+1} \pi\big]} 
\right)^n \right\}.
\end{equation}
\end{proposition}

% ----------------------------------------------------------------------------- %

\begin{proof}
As $A^{(2kn-k)}$ is a circulant matrix, we can utilize the circulant diagonalization theorem to calculate its $n$th power. The latter states that all circulants $c_{ij} = \text{circ}(c_j)$ constructed from an arbitrary $N$-dimensional vector $c_j$ are diagonalized by the same unitary matrix $\mathbf{U}$ with components
\begin{equation}
\label{Eq_Umn}
u_{rs} = \frac{1}{\sqrt{N}} \exp\left[ - \frac{2\pi i}{N_N} (r-1)(s-1) \right] ,
\end{equation}
$r,s \in [1,N]$. Moreover, the $N$ eigenvalues are explicitly given by 
\begin{equation}
E_r(\mathbf{C}) = \sum\limits_{j=1}^{N} c_j \exp\left[ - \frac{2\pi i}{N} (r-1)(j-1) \right] ,
\end{equation}
such that 
\begin{equation}
\label{Eq_cij}
c_{ij} = \sum\limits_{r,s=1}^{N} u_{ir} e_{rs} u^{*}_{sj}
\end{equation}
with $e_{rs} = \text{diag}[ E_r(\mathbf{C}) ] \equiv \delta_{rs} E_r(\mathbf{C})$ and $u^{*}_{rs}$ denoting the complex conjugate of $u_{rs}$.

Using (\ref{Eq_ACMC}), the eigenvalues of $A^{(2kn-k)}$ are
\begin{eqnarray*}
\lefteqn{E_r(A^{(2kn-k)})} \\
& = & \sum\limits_{j=1}^{2kn+1} 
\left\{
\sum\limits_{l=0}^{k} \delta_{j,1+l} + \sum\limits_{l=0}^{k-1} \delta_{j,1+2kn-l} \right\} e^{-2 \pi i \, (r-1)(j-1) / (2kn+1)} \\
& = & \sum\limits_{l=1}^{k+1} e^{-2 \pi i \, (l-1) (r-1) / (2kn+1)} + \sum\limits_{l=2kn+2-k}^{2kn+1} e^{-2 \pi i \, (l-1) (r-1) / (2kn+1)} \\
& = & 1 + 2 \sum\limits_{l=1}^{k} \cos\left[ 2 \frac{r-1}{2kn+1} l \pi \right] \\
& = & \left\{
\begin{array}{ll}
1 + 2 \sin\big[ \frac{k(r-1)}{2kn+1} \pi \big] \cos\big[ \frac{(1+k)(r-1)}{2kn+1} \pi \big] \big/ \sin\big[ \frac{r-1}{2kn+1} \pi \big] & \, r>1 \\
2k+1 & \, r=1 .
\end{array}
\right.
\end{eqnarray*}
Using the product-to-sum identity for trigonometric functions, the last equation can be simplified, yielding
\begin{equation}
E_r(A^{(2kn-k)}) =
\left\{
\begin{array}{ll}
\sin\big[ \frac{2k+1}{2kn+1} (r-1) \pi \big] \big/ \sin\big[ \frac{1}{2kn+1} (r-1) \pi \big] & \quad r>1 \\
2k+1 & \quad r=1 .
\end{array}
\right.
\end{equation}
With this, Eqs.~(\ref{Eq_Umn}) and (\ref{Eq_cij}), one obtains for the elements of the $n$th power of $A^{(2kn-k)}$
\begin{eqnarray*}
\lefteqn{\big(A^{(2kn-k)}\big)_{pq}^n} \\
& = & \sum\limits_{r,s=1}^{2kn+1} u_{pr} E_r^n u^{*}_{rq} \\
& = & \frac{1}{2kn+1} 
\left\{
E_0^n + \sum\limits_{r=2}^{2kn+1} E_r^n e^{- 2 \pi i \, (r-1) (p-q) / (2kn+1)} \right\} \\
& = & \frac{1}{2kn+1} 
\left\{
(2k+1)^n + \sum\limits_{r=1}^{2kn} \left( \frac{\sin\big[ \frac{2k+1}{2kn+1} r \pi \big]}{\sin\big[ \frac{1}{2kn+1} r \pi \big]} \right)^n e^{-2 \pi i \, r (p-q) / (2kn+1)}
\right\} .
\end{eqnarray*}
Taking the trace, finally, proves (\ref{Eq_Msin}).
\end{proof}

Equation (\ref{Eq_Msin}) is remarkable in several respects. First, it provides a general, explicit representation of the sequences of central $(2k+1)$-nomial coefficients in terms of a linearly growing, but finite, sum, thus effectively translating a combinatorical problem into an analytical one. Note specifically $k=1, n\in\mathbb{N}$ yields the sequence of central trinomial coefficients (OEIS A002426), $k=2$ the sequence of central pentanomial coefficients (OEIS A005191), and $k=3$ the sequence of central heptanomial coefficients (OEIS A025012). Secondly, utilizing trigonometric identities, this explicit representation may help to formulate general recurrences not just for coefficients of a given sequence, but between different central multinomial sequences. Moreover, by using different shift parameters $m$ (see proof of Lemma \ref{L_CMCtr}), each $(2k+1)$-nomial coefficient could potentially be represented in a similarly explicit analytical form, thus allowing for a fast numerical calculation of arbitrary $(2k+1)$-nomial coefficients. 

Finally, Proposition \ref{P_Msin} establishes a direct link between central $(2k+1)$-nomial coefficients and the $n$th-degree Fourier series approximation of a function via the Dirichlet kernel $D_k[\theta]$. Using the trigonometric representation of Chebyshev polynomials of the second kind,  
$$
U_{2k}[ \cos(\alpha) ] = \frac{\sin[(2k+1)\alpha]}{\sin[\alpha]} ,
$$
equation (\ref{Eq_Msin}) takes the form
\begin{equation}
\label{Eq_Mu2k}
M^{(2k,n)} = \frac{1}{2kn+1} 
\left\{
(2k+1)^n + \sum\limits_{l=1}^{2kn} \left( U_{2k}\Big[ \cos\big( \tfrac{1}{2kn+1} l \pi \big) \Big] 
\right)^n \right\} .
\end{equation}
Observing $U_{2k}[\cos(l \pi / (2kn+1))] \equiv D_k[2 l \pi / (2kn+1)]$ makes explicit the link between central $(2k+1)$-nomial coefficients and the Dirichlet kernels of fractional angles.

Returning to the original problem by De Volson Wood, however, a definite solution is still at large. Here, relation (\ref{Eq_Msin}) allows so far only for a representation of the combinatorical complexity in terms of an interesting finite analytical construct, thus, in principle, expressing a hard combinatorical problem in a trigonometric framework. 

% ============================================================================= %
% ============================================================================= %

\section*{Acknowledgments}

The authors wish to thank D Zeilberger for valuable comments in the preparation of the manuscript, and OD Little for inspiring comments. This work was supported by CNRS, the European Community (BrainScales Project No. FP7-269921), and \'Ecole des Neurosciences de Paris Ile-de-France.

% ============================================================================= %
% ============================================================================= %

\end{document}